\documentclass[11pt]{amsart}

\usepackage{amsmath, amsthm, amsfonts, amssymb, xy,  mathrsfs, graphicx, lscape, array, booktabs}
\usepackage{enumerate}
\usepackage[usenames, dvipsnames]{xcolor}
\usepackage[margin=1in]{geometry} 
\usepackage{tikz}
\usetikzlibrary{matrix,arrows}

\numberwithin{equation}{section}
\newtheorem{theorem}[equation]{Theorem}

\newtheorem{proposition}[equation]{Proposition}
\newtheorem{lemma}[equation]{Lemma}
\newtheorem{corollary}[equation]{Corollary}

\theoremstyle{definition}
\newtheorem{remark}[equation]{Remark}
\newtheorem{eg}[equation]{Example}
\newenvironment{example}[1][]{\begin{eg}[#1] \pushQED{\qed}}{\popQED \end{eg}}
\newtheorem{defn}[equation]{Definition}
\newenvironment{definition}[1][]{\begin{defn}[#1]\pushQED{\qed}}{\popQED \end{defn}}
\newtheorem{ques}[equation]{Question}

\usepackage[in]{fullpage}
\usepackage[colorlinks=true, pdfstartview=FitV, linkcolor=black, citecolor=black, urlcolor=black]{hyperref}

\newcommand{\ZZ}{\mathbb{Z}}

\newcommand{\NN}{\mathbb{N}}

\renewcommand{\phi}{\varphi}

\newcommand{\comment}[1]{}

\makeatletter
\def\Ddots{\mathinner{\mkern1mu\raise\p@
\vbox{\kern7\p@\hbox{.}}\mkern2mu
\raise4\p@\hbox{.}\mkern2mu\raise7\p@\hbox{.}\mkern1mu}}
\makeatother

\usepackage{multicol}
\usepackage{amsmath}
\usepackage{url}
\usepackage{hyperref}

\newcommand{\taucar}[0]{\widehat{\tau}}

\newcommand{\taubar}[0]{\overline{\tau}}

\newcommand{\fhat}[0]{\widehat{f}}
\newcommand{\ghat}[0]{\widehat{g}}
\newcommand{\sigstar}[0]{\sigma^{*}}
\newcommand{\siginv}[0]{\sigma^{-1}}
\newcommand{\taustar}[0]{\tau^{*}}
\newcommand{\cyclist}[2]{[#1]_{#2}}
\newcommand{\cyclinv}[2]{[#1]^{-1}_{#2}}
\newcommand{\evenlist}[1]{[#1]}
\newcommand{\enelinv}[1]{[#1]^{-1}}
\newcommand{\tand}{\text{ and }}
\newcommand{\tor}{\text{ or }}
\newcommand{\xycycle}{(x \; y)}
\newcommand{\sgroupminus}[2]{S_{#1}\setminus S_{#2}}
\newcommand{\sumSn}[1]{C(#1)}

\newcommand{\finv}[0]{f^{-1}}

\title{Variations on Keeler's Theorem}
\author{A. Bartas, D. Lara, B. Leavitt, B. Vessely}
\begin{document}
\maketitle
\begin{abstract}
The 2010 \emph{Futurama} episode \textit{The Prisoner of Benda} features a mind swapping machine that swaps the minds of two people at a time with the restriction that the same pair of people cannot use the machine more than once. We show that if a machine swaps $n$ people cyclically with the condition that the same group of people cannot use the machine again, we can find a way to get everyone back. We prove our solution is optimal for when $n =3$. We also introduce an infinite variant of the mind swapping machine.  
\end{abstract}

\section{Introduction}
\subsection{Context and Motivation}
In \emph{Futurama's} acclaimed episode, \textit{The Prisoner of Benda}, a machine swaps the minds of two people at once. The episodes features a restriction in which a pair can only use the machine to swap once. A single person, however, can use the machine multiple times with different partners. After a mind switching fiasco, an algorithm is found in the show that gets everyone back into their original body without knowing the sequence of switches that occurred. The solution was written up by one of the show's writers, mathematician Ken Keeler. This algorithm is often referred to as Keeler's Theorem.

Keeler reframed the problem in terms of permutations in the following way. Each brain swap can be described as a transposition in $S_n$. The mind switching fiasco is then a product of transpositions that yields a permutation that is contained in $S_n$. The problem to get everyone back into their original bodies reduces to finding a way to write the inverse of a permutation $\sigma \in S_n$ as a product of distinct transpositions. Keeler's solution uses two external individuals that have not used the machine prior to the mind switching frenzy. It is important to emphasize that the inverse of any permutation in $S_n$ exists; the significance of Keeler's result is that the inverse can be written as a product of distinct transpositions. For example, suppose that Person $1$ and Person $2$ use the machine. If there were no restriction, the inverse of $(1\;2)$ is itself. That is, Person $1$ and Person $2$ would use the machine again to get themselves back into their original bodies. With the constraint, the inverse of the permutation can be written as a product of distinct transpositions as
\begin{equation}
    (1\;2)^{-1} = (x\;y)(2\;x)(1\;y)(2\;y)(1\;x).
\end{equation} 

In 2016, Elder \cite{elder_paper}, \cite{elder_thesis} considered a machine that permutes the minds of $k$ people cyclically. That is, if people $1, 2, \cdots, k$ use the machine, then their minds get permuted as follows: 
$$1 \to 2 \to \cdots \to k  \to 1 .$$
This can be expressed as the $k$-cycle $(1\;2\;3\cdots\;k)$. The restriction that Elder introduced is that once a cycle $\tau$ has occurred, no cycle in the cyclic subgroup $\langle \tau \rangle$ can occur. Elder showed that there is a way to return everyone's mind back to their own body when $k$ is either even or prime. In fact, Elder's proof works for all natural numbers $k \geq 2$.

Evans, Huang, and Nguyen \cite{optimal_paper} explored Keeler's Theorem from a different perspective. They refined Keeler's result by providing a more efficient algorithm that uses the smallest possible number of switches, with the restriction that at least one external must be used in each step. In addition, they also proved that the identity permutation in $S_n$ can can be written as a product of $m$ distinct transpositions for when $m$ is even and when $6 \leq m \leq \binom{n}{2}$.

Our goals are to build upon these results. In our first variation of the problem, the machine permutes $k$ people cyclically, as in Elder's variation. Our work varies from Elder's in that once a set of $k$ people have used the machine, that same set cannot use the machine again. However, each individual can use the machine multiple times. Given any permutation that is obtained from permuting $n$ people cyclically, we provide an algorithm that gets everyone back into their original body. To build on the work of Evans, Huang, and Nguyen, we explore the optimally of our solution for when the machine permutes $3$ people. Finally, we introduce an infinite variation of the machine that swaps the minds of an infinite amount of people. Our main theorems are as follows.

\subsection{Main Results}
We employ a variation similar to that of Elder \cite{elder_paper}. A machine swaps $k$ people cyclically. That is if $a_1, a_2, \cdots, a_k$ use the machine, then the machine permutes their minds as 
$$ a_1 \to a_2  \to \cdots \to a_k \to a_1.$$
 Our restriction is that no set of $k$ people can use the machine more than once. Additionally, we follow Keeler and Elder's assumption in that we don't know which steps were taken to achieve the original mind scrambling, meaning we need to introduce external users when try to revert everyone back into their original bodies. We will refer to these external people as \textit{outsiders} throughout the paper. Two of our main theorems under these assumptions are as follows: 
\begin{theorem}\label{main-n-theorem}
    Let $m,n\in\NN$, $m\ge3$, $n\ge m$. Let $H_m \subseteq S_n$ denote the subgroup generated by all $m$-cycles. If $\sigma\in H_m$, then there exist $\mu_1, \cdots, \mu_r$ such that
    \begin{enumerate}
        \item 
            $\mu_i \in S_{n+d}\setminus S_n$, where
            $d = 
            \begin{cases}
            (m-2) & \text{if } m \text{ is odd}\\
            3(\frac{m}{2}-1) & \text{if } m \text{ is even}
            \end{cases} $
        
        \item $\sigma^{-1} = \mu_1 \cdots \mu_r$, where $\mu_i$ is an $m$-cycle permuting a distinct set of $m$ elements for all $1\leq i \leq r$.
    \end{enumerate}
\end{theorem}
That is, if a collection of $n$ people have used a mind swap machine of size $m \geq 3$, we can revert everyone back with the addition of $m - 2$ outsiders if $m$ is odd and $3 \left( m / 2 - 1 ) \right)$ outsiders if $m$ is even. 
\begin{theorem}\label{optimal}
    Let $\sigma \in S_n$ be written as a product of $r$ disjoint cycles as
    $\sigma = \tau_1 \cdots \tau_r$, where $\tau_i$ is a $k_i$-cycle with $k_i \geq 2$ and 
    $k_1 + \cdots + k_r = n$. Consider outsider $x$ which $S_{n+1}$ acts upon. We can write $\sigma^{-1} \in S_{n+1}\setminus S_n$ as a product of ${(n+r)\over 2}$ 3-cycles, each of which contains $x$. Further, this solution is optimal in that any other construction of $\siginv$ must have at least ${(n+r)\over 2}$ steps. 
\end{theorem}
If $n$ people get scrambled using a machine of size $m = 3$, we can group them up into a $r$ number of disjoint groups where in each group, the individuals have been permutated cyclically. We prove that we need one and only one outsider and show that the minimum number of steps required must be by $\left(  \frac{n + r}{2}\right)$. We later consider a variation of the problem where the machine "cycles" a countably infinite amount of people. The main result shown in this variation is, 

\begin{theorem}
    Let $X$ be a countable infinite set, $\sigma \in S_{fin}(X)$. Then there exist $f_1$ and $f_2$, which are forgetful and retentive respectively, such that $dom(f_1)\neq dom(f_2)$ and $\siginv = f_2 \circ f_1$.
\end{theorem}
This was quite surprising as $\sigma$ could have been created using an arbitrarily large number of swaps and yet its inverse can always be written in 2 steps.

\subsection{Symmetric Group} In this section, we will present fundamental results about the symmetric group $S_n$. For a more detailed recollection, we refer the interested reader to standard references such as Dummit and Foote \cite{dummit_foote_text}. 

We denote the group of permutations of $n$ objects by $S_n$. A \emph{cycle} of length $k$ is a permutation that permutes $k$ objects cyclically. 
$S_n$ is defined to be the set of bijections on the set $\{1,2, \cdots, n\}$, and it forms a group under function composition. The objects of $S_n$ are called \emph{permutations}. For an example, consider the permutation $\sigma \in S_4$ defined as 
$$\sigma = \begin{pmatrix} 
1 & 2 & 3 & 4 \\
4 & 3 & 1 & 2
\end{pmatrix}$$
where $\sigma$ sends $1$ to $4$, $2$ to $3$, and so on. A convenient alternative notation for permutations is called cycle notation. We see that $\sigma$ permutes the objects as 
$$1 \to 4 \to 2 \to 3 \to 1. $$
We can then write $\sigma$ in cycle notation as $\sigma = (1\;4\;2\;3)$. Note that if $\sigma= (1\;4\;2\;3) \in S_{10}$, it still permutes $1 \to 4 \to 2 \to 3 \to 1 $ cyclically, but all other numbers stay fixed. 
A cycle $\tau \in S_n$ of length $k \leq n$ is a function that permutes $k$ objects cyclically. Two cycles are \emph{disjoint} if the numbers moved by one cycle are left fixed by the other. For example, $(1\;3) (2\;5)$ is a product of two disjoint cycles. 
\begin{example}
Let's compute the product 
\begin{equation}
 (1\;5)(1\;2\;3)(1\;4).
 \end{equation}
When multiplying permutations, the permutation on the right is evaluated first. The right most permutation sends $1 \to 4$. The second leaves $4$ fixed, and so does the third. Thus, the permutation sends $1 \to 4$. Next, we evaluate the permutation at $4$. The right most cycle sends $4 \to 1$, the second sends $1 \to 2$, and the third leaves $2$ fixed. Continuing this process, we obtain 
\begin{equation}
    (1\;5)(1\;2\;3)(1\;4) = (1\;4\;2\;3\;5).
\end{equation} 
\end{example}
The example illustrates that when we take a product of non-disjoint cycles, we obtain a product of disjoint cycles (a product of one cycle, in this example). This is true for all permutations. 

\begin{theorem}\label{perm_disjoint_cycles}
    Every permutation of $S_n$ can be written uniquely, up to order, as a product of disjoint cycles.
\end{theorem}
A \emph{transposition} is a cycle of length $2$, such as $(1\;2)$. Every permutation can be written as a product of transpositions.  A permutation is said to be \emph{even} if it can be written as a product of of an even number of transpositions and \emph{odd} if it can be written as a product of an odd number of transpositions. All permutations are either even or odd but not both.
\begin{example}
    $$ (123) = (13)(12) $$
so the cycle $(123)$ is even. 
$$ (1234) = (14)(13)(12)$$
so the cycle $(1234)$ is odd. Let $I$ be the identity permutation. Observe 
$$I = (12)(12)$$
so $I$ is even.
\end{example}
From these examples we can observe that a cycle of odd length is an even permutation and a cycle of even length is an odd permutation. The collection of all even permutation of $S_n$ forms a subgroup called the alternating group $A_n.$ We present  several key results about the alternating group. 
\begin{theorem}\label{A_n_index_generation_lemma}
    $A_n$ is a subgroup of index $2$ and for $n\ge3$, $A_n$ is generated by $3$-cycles. 
\end{theorem}
That is, every permutation in $A_n$ can be written as a product of three cycles. For example, consider $(3 \;4)(1\;2) \in A_n$, then
$$ (3\;4)(1\;2) = (3\;1\;4)(1\;2\;3). $$
Another interesting result is that $A_n$ is a \textit{normal} subgroup of $S_n$, meaning that we can mod out the objects of $S_n$ by $A_n$. The resulting equivalence classes forms a subgroup. This is written as $S_n / A_n$. It turns out that for $n \geq 5$, we have the following interesting result.
\begin{lemma}\label{no-other-normals}
    Let $n\ge 5$. Then if $H$ is a non-trivial normal subgroup of $S_n$, then $H = A_n$.
\end{lemma}

\subsection{Mind Swapping Machines}
The mind swapping machines can be completely described in terms of permutations. Suppose people $a_1, \cdots, a_k$ use the machine. 
The machine is ordered in such a way that person $a_1$ sits on seat $1$ and so on. The machine then performs the permutation $(a_1\;\cdots \;a_k)$.

Given a collection of people whose minds have been permuted by some sequence of uses of the $k$-machine, we can always write that permutation of minds as a product of 
$k$-cycles. From fundamental results about the symmetric group, the resulting permutation is a product of distinct cycles. 

The problem of getting everyone's mind back to their own body is reduced to finding the inverse of $\tau \in S_n $. Every permutation can be written as a product of disjoint cycles. For example, we can write the permutation $\tau \in S_n$ as a product of disjoint cycles where
$$ \tau = \tau_1 \cdots \tau_k$$
with $\tau_i \in S_n$ for all $i$. Observe that in order to find the inverse of $\tau$, we just need to find the inverse of each $\tau_i$:
$$ \tau^{-1} = \tau_k^{-1} \cdots \tau_1 ^{-1}.$$
Under our restrictions, we will not know which sets of people were cycled in the machine to create the original permutation. Thus, we again require that the inverse of each $\tau_i$ must contain an outsider in each step. 

\begin{example}
    Suppose Amy, Bender, and Fry have used the $3$-machine. Regardless of what order they used the machine, we label people as follows. We start with Amy's body. Her body gets the label $a_1$. Suppose Amy's mind got moved to Fry's body. Fry's body is then labeled $a_2$. We then look for Fry's mind and in this case, the only possibility is that it must be in Bender's body.  We label Bender as $a_3$, thus meaning Bender's mind is in Amy's body. The permutation is then written as 
    $$\tau =  (a_1\;a_2\;a_3)$$
    Of course, if the machine allowed the same group of people to use the machine more than once, we could just the machine two more times to get everyone back. Since this is not allowed under our restriction, the goal is then to find outsiders that have not used the machine yet to help revert their minds. That is, we want to write the inverse of $\tau$ as a product of distinct $3$-cycles. 
\end{example}

\section{Keeler's Result}
In the original problem, two people switch minds and realize they cannot switch back. By the end of the episode in which this mind-switching occurs, the characters' minds are all scrambled. Keeler's original solution uses two outsiders to get everyone back into their original bodies. We will reprove this solution in a novel manner. We first note that our solution as proposed in Theorem \ref{solve_single_cycle_m_2} is not exactly the same as Keeler's solution \cite{elder_thesis}. Nevertheless, our algorithm works similarly to Keeler's, while allowing for a clearer proof. 

We call our two outsiders $x$ and $y$, letting $x = n +1$ and $y = n +2$. We define $S_{n+2} \setminus S_n$ to be the collection of permutations that permute $x$ or $y$. 

\begin{theorem} \label{solve_single_cycle_m_2}
Given any cycle $ \tau_k = (a_1 \cdots a_k) \in S_n$, its inverse can be expressed as a product of distinct transpositions in $S_{n+2} \setminus S_n$ in the following form: 
\begin{equation}
    \tau_k^{-1} = (x\; y) \bigg[ \prod_{\ell=2}^{k}(x\; a_\ell) \bigg] (y\; a_1) (y\; a_2) (x\; a_1). 
\end{equation}
\end{theorem}

\begin{proof}
    Let $\tau = (a_1 \cdots a_k) \in S_n$ be a $k$-cycle. We then define 
   \begin{equation} 
        \taubar_k = \prod_{\ell=2}^{k}(x\; a_\ell).
   \end{equation}
We claim that 
\begin{equation}
\tau_k^{-1} = (x\; y)\taubar_k(y\; a_1)(y\; a_2)(x\; a_1),
\end{equation}
or equivalently, 
\begin{equation}
\taubar_k = (x\; y)\tau_k^{-1}(x\; a_1)(y\; a_2)(y\; a_1).
\end{equation}\\
We will prove this using induction on $\taubar_k$. For our base case, in which two people's minds are cycled, we have
\begin{equation}
    \tau_2 = (a_1\; a_2).
\end{equation}
By inspection, we see that 
\begin{equation} 
\tau^{-1}_2 = (x\; y)(x\; a_2)(y\; a_1)(y\; a_2)(x\; a_1).
\end{equation}
We verify that $\taubar_2$ has the desired form by rewriting $\tau^{-1}_2$ as follows:
\begin{equation} 
\tau^{-1}_2 = (x\; y)\taubar_2(y\; a_1)(y\; a_2)(x\; a_1), 
\end{equation}
thus obtaining our desired form,
\begin{equation}
    \taubar_2 = (x\; y)\tau^{-1}_2(x\; a_1)(y\; a_2)(y\; a_1).
\end{equation}

We assume inductively that $\taubar_{k} = (x\; y)\tau^{-1}_k(x\; a_1)(y\; a_2)(y\; a_1)$ holds for some $k \geq 2$. Let $\tau_{k+1} = (a_1 \cdots a_k)$ be a cycle of length $k+1$.
\begin{align}
\taubar_{k+1} &=\bigg[\prod_{\ell=2}^{k+1}(x\; a_\ell)\bigg]\\  
&= \bigg[\prod_{\ell=2}^{k}(x\; a_\ell)\bigg](x\; a_{k+1})\\
&=\taubar_{k}(x\; a_{k+1}).
\end{align}
By our inductive hypothesis,
\begin{align}
\taubar_{k+1} &= (x\; y)\tau_k^{-1}(x\; a_1)(y\; a_2)(y\; a_1)(x\; a_{k+1})\\
&=(x\; y)\tau_k^{-1}(x\; a_1)(x\; a_{k+1})(y\; a_2)(y\; a_1)\\
&=(x\; y)\tau_k^{-1}(a_1\; a_{k+1})(x\; a_1)(y\; a_2)(y\; a_1).
\end{align}
We know that $\tau_k^{-1} = (a_k\; a_{k-1}\cdots a_1)$, so $\tau_k^{-1}(a_1\; a_{k+1}) = \tau_{k+1}^{-1}$. Thus, we have
\begin{equation}
\taubar_{k+1} = (x\; y)\tau_{k+1}^{-1}(x\; a_1)(y\; a_2)(y\; a_1),
\end{equation}
giving us our desired result.
\end{proof}

\begin{remark}
The permutation $\tau^{-1}$ can be expressed as a product of $k+3$ distinct transpositions. 
\end{remark}

\begin{definition}
    Let $\tau = (a_1 \cdots a_k)$ be an arbitrary $k$-cycle in $S_n$. Then we define 
    \begin{equation}\label{car_def}
        \taucar := \taubar (y \; a_1)(y \; a_2)(x \; a_1) = 
    \bigg[ \prod_{\ell=2}^{k}(x \; a_\ell) \bigg]
    (y \; a_1)(y \; a_2)(x \; a_1).\end{equation} 
    
    \let\qed\relax
    \end{definition}

\begin{remark}\label{rewrite_single_cycle_m_2} 
    We note that can use $\taucar$ to rewrite the result from Theorem \ref{solve_single_cycle_m_2} more concisely, which will make subsequent abstraction of the problem clearer:
    given a $k$-cycle $ \tau = (a_1 \cdots a_k) \in S_n$, its inverse can be written in $S_{n+2} \setminus S_n$ as 
    \[ \tau^{-1} = (x \; y)\taucar . \]
    We can rewrite this as $\taucar = (x \; y) \tau^{-1}$, allowing us to derive the following identity:
    \begin{equation}
        \taucar \tau = (x \; y) \tau^{-1} \tau = (x \; y) .
    \end{equation}
\end{remark}

\smallskip

We now look at the case where we have a general permutation $ \sigma \in S_n$ which is not necessarily a single cycle. 

\begin{theorem}

Given $ \sigma \in S_n$, $\sigma^{-1}$ can be written as a product of distinct transpositions in $S_{n+2} \setminus S_n$. Specifically, if 
$\sigma = \tau_1 \cdots \tau_j$ for disjoint cycles $\tau_i \in S_n$ with $1\leq i\leq j$, then $\siginv$ can be written in the form
\[ \sigma^{-1} = 
\begin{cases}
    (xy) \taucar_j \cdots \taucar_1 & \text{if $j$ is odd} \\
    \taucar_j \cdots \taucar_1 & \text{if $j$ is even} . \\ 
\end{cases}
\]

\end{theorem}

\begin{proof}
    Let $\sigma \in S_n$. Theorem \ref{perm_disjoint_cycles} tells us that there exists $j\geq 1$ such that $\sigma$ can be written as a product of $j$ disjoint cycles:   
\begin{equation} \sigma = \tau_1 \cdots \tau_j = \prod_{i=1}^j \tau_i. \end{equation}
    
    We first have the case when $j$ is odd.
    Consider 
    \begin{equation}\gamma = \xycycle \taucar_j \cdots \taucar_1 = \xycycle \prod_{i=0}^{j-1} \taucar_{j-i}. \end{equation}
    We then have
    \begin{align}
        \gamma \sigma 
            &= \xycycle \prod_{i=0}^{j-1} \taucar_{j-i} \; \prod_{i=1}^j \tau_i \\
            &= \xycycle \bigg[ \prod_{i=0}^{j-2} \taucar_{j-i} \bigg] 
            \taucar_1 \tau_1 \bigg[ \prod_{i=2}^{j} \tau_i \bigg] \\
            &= \xycycle \bigg[ \prod_{i=0}^{j-2} \taucar_{j-i} \bigg] 
            \xycycle \bigg[ \prod_{i=2}^{j} \tau_i \bigg] .
    \end{align}
    Because $\tau_i \in S_n$ for all $i \geq 1$, $\tau_i$ does not contain the elements $x \tor y$ at any point. This means that $\tau_i \tand (x \; y)$ are disjoint, so we can commute them:
    \begin{equation}
        \gamma \sigma = \xycycle \bigg[ \prod_{i=0}^{j-2} \taucar_{j-i} \bigg] 
            \bigg[ \prod_{i=2}^{j} \tau_i \bigg] \xycycle .
    \end{equation}
    We then repeat this process, cancelling each $\taucar_i$ with $\tau_i$ to get $\xycycle$, then commuting $\xycycle$ to the far right of the equation. We will do this $j$ times, once for each $\tau_i$, leaving us with
    \begin{align}
        \gamma \sigma &= \xycycle \xycycle^j \\
            &= \xycycle^{j+1} \\  
            &= \left( \xycycle^{2} \right)^{(j+1)/2} \\
            &= (e)^{(j+1)/2} \\
            &= e
    \end{align}
    So, for all $j\geq 1$ with $j$ odd, we have shown that $\gamma = \xycycle \prod_{i=0}^{j-1} \taucar_{j-i} = \siginv.$ \\

    We now look at the case where $j$ is even. We still have $\sigma = \tau_1 \cdots \tau_j = \prod_{i=1}^j \tau_i$. Consider the product
    \begin{equation} \theta = \taucar_1 \cdots \taucar_j = \prod_{i=0}^{j-1} \taucar_{j-i}. \end{equation}
    We then have
    \begin{align}
        \theta \sigma &= \prod_{i=0}^{j-1} \taucar_{j-i} \prod_{i=1}^j \tau_i \\ 
            &= \bigg[ \prod_{i=0}^{j-2} \taucar_{j-i} \bigg] \taucar_1 \tau_1
                \bigg[ \prod_{i=2}^j \tau_i \bigg] \\
            &= \bigg[ \prod_{i=0}^{j-2} \taucar_{j-i} \bigg] \xycycle
                \bigg[ \prod_{i=2}^j \tau_i \bigg] \\
            &= \bigg[ \prod_{i=0}^{j-2} \taucar_{j-i} \bigg]
                \bigg[ \prod_{i=2}^j \tau_i \bigg] \xycycle ,
    \end{align}
    where we've used the same logic as in the odd case to commute $\xycycle$. We repeat this process $j-1$ more times, once for each $\tau_i$ remaining in the right product. This gives us x
    \begin{align}
        \theta \sigma &= \xycycle^j \\
            &= \left( \xycycle^2 \right)^{j/2} \\ 
            &= \left( e \right)^{j/2} \\ 
            &= e.
    \end{align}
    Thus, we have 
    $\theta = \prod_{i=0}^{j-1} \taucar_{j-i} = \taucar_j \cdots \taucar_1 = \sigma^{-1}$,
    as desired. 
\end{proof}

\section{Machines of Larger Sizes}
Now suppose the machine swaps $m$ people at at time, cyclically. That is, 
$$ 1 \to 2 \to 3 \to \cdots m \to 1$$
To extend the rules from Keeler's original problem, we will assume that the same group of people cannot use the machine more than once but that individually, each person can use the machine multiple times. Reframed in terms of permutations, we have our main result for this section. 

\begin{theorem}\label{main-n-theorem}
    Let $m,n\in\NN$, $m\ge3$, $n\ge m$. Let $H_m \subseteq S_n$ denote the subgroup generated by all $m$-cycles. If $\sigma\in H_m$, then there exist $\mu_1, \cdots, \mu_r$ such that
    \begin{enumerate}
        \item 
            $\mu_i \in S_{n+d}\setminus S_n$, where
            $d = 
            \begin{cases}
            (m-2) & \text{if } m \text{ is odd}\\
            3(\frac{m}{2}-1) & \text{if } m \text{ is even}
            \end{cases} $
        
        \item $\sigma^{-1} = \mu_1 \cdots \mu_r$, where $\mu_i$ is an $m$-cycle permuting a distinct set of $m$ elements for all $1\leq i \leq r$.
    \end{enumerate}
    
\end{theorem}

We will refer to $S_n$ as the group that permutes $A = \{a_1 \cdots a_n\}$ and $S_{n+d}$ as the group that permutes $A\cup \{x_1\cdots x_{d}\}$. So an $m$-cycle $\mu \in S_{n+d}\setminus S_n$ is a permutation that permutes at least one element in $\{x_1\cdots x_{d}\}$.
The part of this theorem that distinguishes our problem from Elder's is (2). As an example:
$$(a_1\; a_2\; a_3)^{-1} = (a_1\; a_3\; a_2) = (a_1\; a_2\; a_3\; x_2\; x_1)(a_3\; a_1\; a_2\; x_1\; x_2)$$
is a perfectly acceptable solution under Elder's variation as \[\langle(a_1\; a_2\; a_3\; x_2\; x_1)\rangle \neq \langle(a_3\; a_1\; a_2\; x_1\; x_2)\rangle.\] But for our purposes we cannot use both of these cycles, as each cycle permutes the same 5 elements. It is worthwhile to note that any solution under our restriction can be used under Elder's restriction (perhaps not as optimally). The case in which $\sigma$ is a $3$-cycle is of particular interest to us. We will prove this case in the following lemma.

\begin{lemma}\label{three-cycle-lemma}
     Let $m,n\in\NN$, $m \geq 3$, $n\ge m$. If $\tau \in S_n$ is a $3$-cycle, then there exist $\mu_1, \mu_2$ where $\mu_i \in S_{n + (m-2)}\setminus S_n$ such that 
    \begin{enumerate}
        \item $\tau^{-1} = \mu_1\mu_2$
        
        \item For all $1\leq i \leq 2$, $\mu_i$ is an $m$-cycle permuting a distinct set of $m$ elements
    \end{enumerate}
\end{lemma}

\begin{proof}
    Let $\tau = (a_1\; a_2\; a_3)$. Then $$\tau^{-1} = (a_1\; a_3\; x_1 \cdots x_{m-2})(a_3\; a_2\; x_{m-2} \cdots x_1)$$
\end{proof}

 It is worthwhile noting that in this lemma, the solution does not depend on the parity of $m$. 

\begin{lemma}\label{H_m_normal_lemma}
    Let $m,n\in\NN$, $m\ge3$, $n\ge m$. $H_m \subseteq S_n$ is a normal subgroup.
\end{lemma}

\begin{proof}
    Let $\tau = (a_1\cdots a_m) \in S_n$ be an $m$-cycle and thus a generator of $H_m$. For all $\sigma \in S_n$,
    $$\sigma \tau \siginv = (\sigma(a_1) \cdots \sigma(a_m)).$$
    Since $\sigma$ is a bijection, there exists $\tau' \in H_m$ such that $$\tau' = (\siginv(a_1)\cdots\siginv(a_m)).$$ 
    Thus, $\sigma\tau'\siginv = \tau$. It is clear that conjugation permutes the generators of $H_m$, hence $\sigma H_m \siginv = H_m$.
\end{proof}

We will prove theorem \ref{main-n-theorem} in even and odd sections. 

\subsection{Case 1: $m$ is odd}  We will use the following without proof.

\begin{lemma}\label{H_m_A_n_odd}
    Let $m,n\in\NN$, $n\ge m \geq 3$. $H_m = A_n$ if and only if $m$ is odd.
\end{lemma}

\begin{proof}
    $\Rightarrow$ Assume that $m$ is even, and let $\tau$ be an $m$-cycle. We know that $\tau$ is certainly in $H_m$ by definition. Since $\tau$ is a cycle of even length, it is an odd permutation. Hence $\tau \notin A_n$. \\
    $\Leftarrow$
    If $m=3$, the result follows from Lemma \ref{A_n_index_generation_lemma}.
    Suppose $m\geq 5$. From Lemma \ref{H_m_normal_lemma} we know $H_m$ is normal. Thus the result follows from Lemma \ref{no-other-normals}.
\end{proof}

This is everything we need to prove the case where the machine switches an odd number of people. To introduce some notation, we denote the elements permuted by $S_{n+(m-2)}\setminus S_n$ in a list by: 
\begin{align*}
    &\evenlist{x} = x_1\ x_2 \cdots x_{m-2}\\
    &\enelinv{x} = x_{m-2} \cdots x_2\ x_1\\
\end{align*}

So from lemma \ref{three-cycle-lemma} if $\tau = (a_1\; a_2\; a_3) \in S_n$, then its inverse in $S_{n+(m-2)}\setminus S_n$ under our new notation can be written as
$$\tau^{-1} = (a_1\; a_3\; [x])(a_3\; a_2\; \enelinv{x})$$

\begin{theorem}\label{odd_thm}
    Let $m,n\in\NN$, with $m$ odd, $n\ge m \ge3$. If $\sigma\in H_m$, then there exist $\mu_1, \cdots, \mu_r \in S_{n + (m-2)}\setminus S_n$ for some $r\ge 1$ such that 
    \begin{enumerate}
        
        \item $\sigma^{-1} = \mu_r \cdots \mu_1$
        
        \item For all $1\leq i \leq r$, $\mu_i$ is an $m$-cycle permuting a distinct set of $m$ elements
    \end{enumerate}
    
\end{theorem}

\begin{proof}
    By the previous lemma $H_m = A_n$ so we need only consider the case for which $\sigma$ is even.
    We can write $\sigma$ as a product of disjoint cycles. $$\sigma = \tau_1\cdots\tau_r$$
    Let $\tau_i$ have length $k_i$. If $k_i$ is odd for all $i$, then $\tau_i \in A_n$ for all $i$. Let $$\tau_i = (a_{i,1}\cdots a_{i,k_i}).$$
    Since $k_1$ is odd, we can write
    $$(a_{i,1}\cdots a_{i,k_i}) = (a_{i,1}\; a_{i,2}\; a_{i,3})(a_{i,3}\; a_{i,4}\; a_{i,5})\cdots(a_{i,k_i-2}\; a_{i,k_i-1}\; a_{i,k_i}).$$
    By lemma \ref{three-cycle-lemma}, we know we can write the inverses of each of these $3$-cycles as a product of two $m$-cycles in $S_{n + (m-2)}\setminus S_n$. Additionally, any given $3$-cycle is disjoint from all others with exception to the ones next to it. Therefore, it is sufficient to check that the inverse of $ (a_{i,1}\; a_{i,2}\; a_{i,3})(a_{i,3}\; a_{i,4}\; a_{i,5})$ written as $m$-cycles satisfies \textit{(2)} in Theorem \ref{odd_thm}. We proceed using lemma \ref{three-cycle-lemma}
    \begin{align*}
        &(a_{i,1}\; a_{i,2}\; a_{i,3})^{-1} = (a_{i,1}\; a_{i,3}\; a_{i,2}) = (a_{i,1}\; a_{i,3}\; \evenlist{x})(a_{i,3}\; a_{i,2}\; \enelinv{x})\\
        &(a_{i,3}\; a_{i,4}\; a_{i,5})^{-1} = (a_{i,3}\; a_{i,5}\; a_{i,4}) = (a_{i,3}\; a_{i,5}\; \evenlist{x})(a_{i,5}\; a_{i,4}\; \enelinv{x}).
    \end{align*}
    All $m$-cycles above permute a distinct set of $m$ elements, so we can define $$\mu_i = (a_{i,k_i-2}\; a_{i,k_i}\;\evenlist{x})(a_{i,k_i}\; a_{i,k_i-1}\; \enelinv{x})\cdots(a_{i,1}\; a_{i,3}\; \evenlist{x})(a_{i,3}\; a_{i,2}\; \enelinv{x}).$$
    Where $\mu_i = \tau_i^{-1}$ hence
    $$\siginv = \mu_r\cdots\mu_1.$$
    Now suppose there exists $i$ such that $k_i$ is even. Since $\sigma \in A_n$, there must also exist $j\neq i$ such that $k_j$ is even. Let
 
    \begin{align*}
        &\tau_i = (a_1\cdots a_{k_i}) = (a_1\cdots a_{k_i-1})(a_{k_i-1}\; a_{k_i})\\
        &\tau_j = (b_1\cdots b_{k_j}) = (b_1\cdots b_{k_j-1})(b_{k_j-1}\; b_{k_j})
    \end{align*}

    So to fix $\tau_i$ and $\tau_j$, we show it is possible to fix their product
    \begin{align*}
        \tau_i \tau_j &= (a_1\cdots a_{k_i-1})(a_{k_i-1}\; a_{k_i})(b_1\cdots b_{k_j-1})(b_{k_j-1}\; b_{k_j})\\
        &= (a_1\cdots a_{k_i-1})(b_1\cdots b_{k_j-1})(a_{k_i-1}\; a_{k_i})(b_{k_j-1}\; b_{k_j})
    \end{align*}

    We have shown how to write the inverse for a cycle of odd length. Thus for this case it is sufficient to show we can write the inverse of $(a_{k_i-1}\; a_{k_i})(b_{k_j-1}\; b_{k_j})$ under our restrictions. 
    $$(a_{k_i-1}\; a_{k_i})(b_{k_j-1}\; b_{k_j}) = (a_{k_i-1}\; a_{k_i}\; b_{k_j-1})(a_{k_i}\; b_{k_j-1}\; b_{k_j})$$
    So we have
    $$((a_{k_i-1}\; a_{k_i})(b_{k_j-1}\; b_{k_j}))^{-1} = (a_{k_i}\; b_{k_j}\; b_{k_j-1})(a_{k_i-1}\; b_{k_j-1}\; a_{k_i}).$$

    And from our previous work,     
    \begin{align*}
        &(a_{k_i}\; b_{k_j}\; b_{k_j-1}) = (a_{k_i}\; b_{k_j}\; \evenlist{x})(b_{k_j}\; b_{k_j-1}\; \enelinv{x})\\
        &(a_{k_i-1}\; b_{k_j-1}\; a_{k_i}) = (a_{k_i-1}\; b_{k_j-1}\; \evenlist{x})(b_{k_j-1}\; a_{k_i}\; \enelinv{x})
    \end{align*}

    Where each $m$-cycle permutes a distinct set of elements. 
\end{proof}

\subsection{Case 2: $m$ is even}

\begin{lemma}
    Let $m\ge2$ be an even number and $n\ge m$. Then $H_m = S_n$.
\end{lemma}

\begin{proof}

    We know that $S_n$ is generated by transpositions so it is sufficient to show that all transpositions lie in $H_m$. Let $(x\;y) \in S_n$. Notice that one can obtain an $(m-1)$-cycle from the product of two $m$-cycles as follows:
    $$(y\; a_1\; x\; a_2\cdots a_{m-2})(y\; x\; a_1\ a_2\cdots a_{m-2}) = (y\; a_2\; a_4\; a_6\cdots a_{m-2}\; a_1\; a_3\cdots a_{m-3}).$$
    Then we can obtain our transposition by taking the product of this $(m-1)$-cycle with a certain $m$-cycle: 
    \begin{align*}
        &(x\; y\; a_{m-3}\cdots a_3\; a_1\; a_{m-2}\cdots a_6\; a_4\; a_2)(y\; a_2\; a_4\; a_6\cdots a_{m-2}\; a_1\; a_3\cdots a_{m-3})\\
        &=(x\; y)(y\; a_{m-3}\cdots a_3\; a_1\; a_{m-2}\cdots a_6\; a_4\; a_2)(y\; a_2\; a_4\; a_6\cdots a_{m-2}\; a_1\; a_3\cdots a_{m-3})\\
        &=(x\; y).
    \end{align*}
    Thus $(x\; y) \in H_m$. 
\end{proof}

\begin{lemma}\label{even-fix-transpostion}

    Let $m,n\in\NN$, with $m$ even, $n\ge m \ge 4
    $. Let $\sigma\in S_n$ be a transposition. Then there exist $\mu_1, \mu_2, \mu_3 \in S_{n + 3(\frac{m}{2}-1)}\setminus S_n$
    
    \begin{enumerate}
        
        \item $\sigma^{-1} = \mu_1 \mu_2 \mu_3$
        
        \item For all $1\leq i \leq 3$, $\mu_i$ is an $m$-cycle permuting a distinct set of $m$ elements
    \end{enumerate}

\end{lemma}

To introduce some notation, we denote the elements permuted by $S_{n + 3(\frac{m}{2}-1)}\setminus S_n$ in three groups. We denote them in a list by 
\begin{align*}
    &\evenlist{w} = x_1\ x_2 \cdots x_{\frac{m}{2}-1}\\
    &\enelinv{w} = x_{\frac{m}{2}-1} \cdots x_2\ x_1\\
    &\evenlist{y} = y_1\ y_2 \cdots y_{\frac{m}{2}-1}\\
    &\enelinv{y} = y_{\frac{m}{2}-1} \cdots y_2\ y_1\\
    &\evenlist{z} = z_1\ z_2 \cdots z_{\frac{m}{2}-1}\\
    &\enelinv{z} = z_{\frac{m}{2}-1} \cdots z_2\ z_1.
\end{align*}

Notice that $|\evenlist{w}| + |\evenlist{y}| + |\evenlist{z}| = 3(\frac{m}{2}-1)$.
\begin{proof}
    Let $\sigma = (a_1\; a_2)$. 
    Then $$\siginv = (a_1\ \enelinv{z}\ \enelinv{y}\ a_2)(a_1\ \enelinv{w}\ a_2\ \evenlist{z})(\evenlist{w}\ a_1\ \evenlist{y}\ a_2)$$
\end{proof}
    
\begin{theorem}
    Let $m,n\in\NN$, with $m$ even, $n\ge m \ge 4
    $. If $\sigma\in H_m = S_n$, then there exist $\mu_1, \cdots, \mu_r$ where $\mu_i \in S_{n + 3(\frac{m}{2}-1)}\setminus S_n$ such that 
    \begin{enumerate}
        
        \item $\sigma^{-1} = \mu_1 \cdots \mu_r$
        
        \item For all $1\leq i \leq r$, $\mu_i$ is an $m$-cycle permuting a distinct set of $m$ elements
    \end{enumerate}
    
\end{theorem}

\begin{proof}
    Since $\sigma$ can be written as a product of disjoint cycles, it is sufficient to prove for the case when $\sigma$ is a cycle. Suppose $\sigma$ has odd length. Then $\sigma \in A_n$. Thus it can be written as a product of three cycles and by lemma \ref{three-cycle-lemma}, we are done. Suppose $\sigma$ has even length $\sigma = (a_1\; a_2 \cdots a_k).$
    $$(a_1\; a_2 \cdots a_k) = (a_1\; a_2 \cdots a_{k-1}) (a_{k-1}\; a_k).$$ 
    $(a_1\; a_2\cdots a_{k-1})$ is a cycle of odd length and by lemma \ref{even-fix-transpostion}, we can write the inverse of the transposition.
\end{proof}

\begin{example}
Let $\sigma = (a_1\; a_2\; a_3)(a_4\; a_5\; a_6\; a_7) \in S_n$. Then 
\begin{align*}
    (a_1\; a_2\; a_3)^{-1} &= (a_1\; a_3\; \evenlist{x})(a_3\; a_2\; \enelinv{x})\\
    (a_4\; a_5\; a_6\; a_7)^{-1} &= ((a_4\; a_5\; a_6)(a_6\; a_7))^{-1}\\
    &=(a_6\ \enelinv{z}\ \enelinv{y}\ a_7)(a_6\ \enelinv{w}\ a_7\ \evenlist{z})(\evenlist{w}\ a_6\ \evenlist{y}\ a_7)(a_4\; a_6\; \evenlist{x})(a_6\; a_5\; \enelinv{x})
\end{align*}

\end{example}

\section{Optimality for 3-machines}


In our context, an optimal solution is defined to have the least possible number of machine uses. Evans et al. previously explored the problem of optimizing the solution to Keeler's Theorem using the 2-machine \cite{optimal_paper}. Keeler's original 2-machine solution has been shown to be optimal only when the number of disjoint cycles $r$ that make up the original permutation is equal to 1 or 2. Evans et al. found an optimal solution for the case in which $r \geq 3$ \cite{optimal_paper}. We first expand this question by looking at the optimal solution for 3-machines. 

\begin{definition}\label{objects_def}
    For the rest of this section, we define
    $\phi = \{ a_1, a_2, \cdots, a_n \}$ to be the set of elements that $S_n$ acts on. Specifically, if $S$ denotes the symmetric group, we have $S_n = S(\phi)$.
\end{definition}

\begin{definition}\label{outsiders_set_def}
    We denote the set of $d$ outsiders as $X = \{x_1, \cdots x_d\}$. We then define $\phi \cup X$ to be the set of elements that $A_{n+d}$ acts on 
    (in other words $A_{n+d} = A(\phi \cup X)$).
\end{definition}

\begin{definition}\label{sumSn_def}
    Let $\alpha \in \sgroupminus{n+1}{n}$ be written as a product of disjoint cycles 
    $\alpha = \tau_1 \cdots \tau_r$, where each $\tau_i$ with $1 \leq i \leq r$ has length $k_i\geq 2$. Then,
    \begin{equation} \sumSn{\alpha} 
        = \sum_{i=1}^r \sum_{
                            \substack{
                            a_j \in \phi \\ 
                            \tau_i(a_j) \neq a_j
                            }
                            } 1  .
    \end{equation}
    In words, $\sumSn{\alpha}$ is the number of times that any element from $\phi$ shows up in a $\tau_i$ cycle. 
\end{definition}

\begin{lemma}\label{3_machine_count}
    Let $\beta$ be a product of 3-cycles in $\sgroupminus{n+1}{n}$ such that each 3-cycle contains $x$. Then the number of 3-cycles composing $\beta$ is $\frac{C(\beta)}{2}$. 
\end{lemma}

\begin{proof}
    Every 3-cycle is in $\sgroupminus{n+1}{n}$ and contains $x$. This means that exactly 2 elements of any given cycle in $\beta$ must be in $\phi$. Thus, $\sumSn{\beta}$ is two times the number of cycles, giving the number of cycles as $\sumSn{\beta} / 2$. 
\end{proof}

\begin{remark}\label{m_less_than_n}
    Let $\sigma \in A_n$ be written as a product of $r$ disjoint cycles,
    $\sigma = \tau_1 \cdots \tau_r$, where $\tau_i$ is a $k_i$-cycle with $k_i \geq 2$. We will assume for all of this section that $k_1 + k_2 + \cdots + k_r = n$. This assumption does not present an isuse, because if $k_1 + k_2 + \cdots + k_r = \ell < n$, there are $n-\ell$ objects that remain unpermuted by $\sigma$ and thus don't need fixing, while the $\ell$ elements that are permuted can be fixed by mimicking the following arguments with $n$ replaced by $\ell$ \cite{optimal_paper}.
\end{remark}

\begin{theorem}\label{3_machine_optimal}
    Let $\sigma \in A_n$ be written as a product of $r$ disjoint cycles,
    $\sigma = \tau_1 \cdots \tau_r$, where $\tau_i$ is a $k_i$-cycle with $k_i \geq 2$ and 
    $k_1 + \cdots + k_r = n$. Consider outsider $x$ which $S_{n+1}$ acts on. We can write $\sigma^{-1} \in S_{n+1}\setminus S_n$ as a product of ${(n+r)\over 2}$ 3-cycles, each of which contains $x$. Further, this solution is optimal in that any construction of $\siginv$ must have at least ${(n+r)\over 2}$ steps. 
\end{theorem}

\begin{proof}
    Let $\sigma = \tau_1 \cdots \tau_r$, with each $\tau_i$ a disjoint cycle. 

    We rewrite the equation as $\sigma = \tau_1 \cdots \tau_q \tau_{q+1} \cdots \tau_r$, with $\tau_1 \cdots \tau_q$ representing cycles of odd length and $\tau_{q+1} \cdots \tau_r$ representing cycles of even length. Even-length cycles have odd parity, and thus must occur in pairs. Corresponding to an arbitrary pair of even-length cycles $\tau_v = (b_1 \cdots b_{k_1})$ and $\tau_w = (c_1 \cdots c_{k_2})$,
    we define
    \[ 
    G(\tau_v, \tau_w) = (c_1\; c_{k_2}\; x) \bigg[ \prod_ { \substack { \ell=2 \\ \text{even} } } ^{k_2-2} (c_{\ell + 1} \; c_\ell \; x) \bigg] (b_1\; b_{k_1}\; x) \bigg[ \prod_ { \substack { \ell=2 \\ \text{even} } } ^{k_1-2} (b_{\ell + 1} \; b_\ell \; x) \bigg] (c_{k_2} \; b_{k_1} \; x). 
    \]
    Corresponding to an arbitrary odd-length cycle $\tau_u = (a_1 \cdots a_k)$, we define
    \[
    F(\tau_u) = \bigg[ \prod_ { \substack { \ell=2 \\ \text{even} } } ^{k-1} (a_{\ell + 1} \; a_\ell \; x) \bigg] (a_2 \; a_1 \; x). 
    \]
    Then we set
    \[
    \siginv = F(\tau_1) \cdots F(\tau_q) \cdot G(\tau_{q+1}, \tau_{q+2}) \cdots G(\tau_{r-1}, \tau_r).
    \]

    It can be verified that $\siginv$ is indeed the inverse of $\sigma$ and that it is made up of $\frac{(n+r)}{2}$ steps. \\

    We will now prove that if $\siginv \in S_{n+1} \setminus S_n$ is written as a product of 3-cycles, it must have at least $\frac{(n+r)}{2}$ elements of $\phi$.  By Lemma \ref{3_machine_count}, we just need to show that $\sumSn{\siginv} \geq n+r$. 
    
    We first note that every element of $\phi$ has had its mind permuted away from its body. This is true because we know that the lengths of the $\tau_i$ cycles sum to $n$ and that all $\tau_i$ cycles are disjoint. In order to return each element's mind to its body, the element must therefore be permuted again in the inverse, giving $\sumSn{\siginv} \geq n$. 

    For the following argument, we also specify that every $\tau_i$ cycle is written with $x$ as the rightmost entry.

    We will now show that every cycle $\tau_i$ must have one element from $S_n$ that shows up at least twice in $\siginv$. Since there are $r$ cycles, this will mean that $\sumSn{\siginv} \geq n + r$. 
    
    Observe cycle $\tau_i \in \sigma$ for some $1 \leq i \leq r$, and let $\gamma$ be the set of elements of $\phi$ that are permuted by $\tau_i$. Next, find the right most element of $\gamma$ in $\siginv$, and call the cycle that this right-most element lies in $T$. 

    There are three cases for how $T$ could look, which correspond to the (sum of binomial coefficients) $\binom{2}{1} + \binom{2}{2}$ ways to have $x$ and at least one $\gamma$ element in $T$. First, it could look like 
    $(a_{i,1} \; a_{i,2} \; x)$ for some elements $a_{i,1}, a_{i,2} \in \gamma$. It could also look like $(a_{i,3} \; u \; x)$ for some $u$ in a different $\tau$ cycle and $a_{i,3} \in \gamma$. Finally, it could look like 
    $(v \; a_{i,4} \; x)$ for some $v$ in a different $\tau$ cycle and $a_{i,4} \in \gamma$.

    For our first case where $T = (a_{i,1} \; a_{i,2} \; x)$, we note that $a_{i,1}$ gets whatever mind was in $x$'s body. Because $T$ is the farthest right (and thus chronologically first) cycle in $\siginv$ containing elements of $\gamma$, $x$ must have its own mind or the mind of an element in $\phi \setminus \gamma$. So, $a_{i,1}$ gets a mind that is not from $\gamma$. This means that $a_{i,1}$ definitely does not have its own mind at this point, so it must be used again in another 3-cycle in $\siginv$ to get its mind back.

    We then look at the case that $T = (a_{i,3} \; u \; x)$. Notice that $a_{i,3}$ gets the mind that was in $x$'s body. For the same reasoning as before, $a_{i,3}$ must get used again in another $3$-cycle in $\siginv $ to get its mind back. 

    If $T = (v \; a_{i,4} \; x)$, then $a_{i,4}$ gets the mind that $v$ had. Since $T$ is the farthest right 3-cycle in $\siginv$ containing elements of $\gamma$, $v$'s body must have had $x$'s mind or a mind from $\phi \setminus \gamma$. So, $a_{i,4}$ must be used at least twice in $\siginv$. 
    
    We have shown that for an arbitrary cycle $\tau_i \in \sigma$, at least one element of $\tau_i$ must show up twice in $\siginv$. This is true for all $r$ cycles in $\sigma$, so $\sumSn{\siginv} \geq n + r$. Lemma \ref{3_machine_count} tells us that $\siginv$ must have at least $\frac{(n+r)}{2}$ steps, completing our proof. 
\end{proof}

\begin{lemma}\label{3_C_count}
    Let $\sigma \in S_n$ be written as a product of $r$ disjoint cycles as
    $\sigma = \tau_1 \cdots \tau_r$, where $\tau_i$ is a $k_i$-cycle with $k_i \geq 2$ and $k_1 + \cdots + k_r = n$. For $d\geq 1$, $\siginv \in S_{n+d} \setminus S_n$ satisfies $C(\siginv) \geq n+r$. 
\end{lemma}

\begin{proof}
    This is a generalization of the argument made in Theorem \ref{3_machine_optimal}, but with any number of outsiders. 
    For an arbitrary cycle $\tau_i \in \sigma$, define $T \tand \gamma$ as in Theorem \ref{3_machine_optimal}. There are two general cases: there is an element of $\gamma$ in the far left position of $T$, or not. Assume that all 3-cycles in $\siginv$ are written with any outsiders as far right as possible.

    If there is an element of $\gamma$ in the far left position of $T$, then $T$ looks like $(a_{i,1} \; b \; x_{j,1})$ for $a_{i,1} \in \gamma$, $b$ arbitrary, and  $x_{j,1} \in \{x_1, \cdots, x_d\}$. By construction, $x_{j,1}$ does not contain a mind from $\gamma$. So, $a_{i,1}$ gets a mind from outside of $\gamma$, which cannot be its own mind. This means $a_{i,1}$ must be used at least twice in $\siginv$, and since this was true for all cycles, $C(\siginv) \geq n+r$. 

    In the case that there is not an element of $\gamma$ in the far left position of $T$, then $T$ looks like $(c \; a_{i,2} \; x_{j,2})$ for $c \notin \gamma \cup 
    \{x_1, \cdots, x_d\}$, $a_{i,2} \in \gamma$, and $x_{j,2} \in \{x_1, \cdots, x_r\}$. By construction, $c$ cannot contain a mind from $\gamma$. So, $a_{i,2}$ gets this mind from outside of $\gamma$, again giving $C(\siginv) \geq n+r$. 
\end{proof}

\begin{corollary}\label{3_optimal_people}
    Let $\sigma \in S_n$ be written as a product of $r$ disjoint cycles as
    $\sigma = \tau_1 \cdots \tau_r$, where $\tau_i$ is a $k_i$-cycle with $k_i \geq 2$ and $k_1 + \cdots + k_r = n$. Then, $\siginv \in S_{n+d} \setminus S_n$ will take at least $\frac{(n+r)}{2}$ steps for all $d \geq 2$. 
\end{corollary}

\begin{proof}


    Lemma \ref{3_C_count} says that $\siginv$ must satisfy $C(\siginv) \geq n+r$. But, having at least one outsider per 3-cycle (regardless of which outsider it is) means that there can be at most 2 elements of $\phi$ in each 3-cycle. This means $\siginv$ must have at least $\frac{(n+r)}{2}$ steps, as desired. 
\end{proof}

\section{An Infinitely Large Machine}

Here we present an infinite analogy of the machine. Consider a machine that takes a set of people enumerated by the natural numbers and that permutes their minds as follows. 
\begin{align*}
        1 &\to 2 \\
        2 &\to 3 \\
           \vdots 
\end{align*}
 
This can be denoted as a function $f: \NN \to \NN$ defined as $f(n) = n+1$. 

\begin{align*}
    f = 1\to2\to3\cdots\\
\end{align*}

We extend the rules to this infinite case as follows. 
\begin{enumerate}
    \item Each seat must be occupied in order for the machine to work. 
    \item Once a collection of people have used the machine, that same collection cannot use the machine again. 
\end{enumerate}
This case differs since we are not working with bijective functions so there may be bodies that remain with no machine. Swaps that do this will be called \textit{forgetful}. In the example above, no one will occupy body $1$. If the machine is only able to do forgetful swap, we would have no hope of fixing everyone as in every swap as someone will always remain mindless. Therefore, we give the machine another capability, to change direction in which minds are swapped:

\begin{align*}
     2 &\to 1 \\
        3 &\to 2 \\
           \vdots 
\end{align*}

This can be denoted as a function $g: \NN \to \NN$ defined as $g(n) = n-1$. Written in one line, 

\begin{align*}
    g = \cdots3\to2\to1
\end{align*}
it is more clear to see why we say the machine can "change direction". We will only use this kind of swapping when the first person (in this case body 1) has no mind occupying it (else body 1 would end up with two minds occupying it). We will call this type of swapping \textit{retentive}, as in no bodies remain mindless. While the swapping above would solve our original problem, it is not 
allowed under our rules as it swaps the same set of people. To summarize, here is an outline of our rules. 
\begin{enumerate}
    \item Each seat must be occupied in order for the machine to work. 
    \item Mindless bodies will remain on the machine. 
    \item Once a collection of people have used the machine, that same collection cannot use the machine again. Individuals can use the machine many times. 
    \item The machine can do \textit{forgetful} swaps. 
    \item The machine can do  \textit{retentive} swaps.
\end{enumerate}

We show the following results:
\begin{lemma}
    Let $f_0: \NN \to \NN$ defined as $f_0(n) = n+1$. Then there exist $f_1$, $f_2$, and $f_3$ where $f_1$, $f_3$ are retentive and $f_2$ is forgetful, such that $dom(f_i)\neq dom(f_j)$ for all $i\neq j$, and $\finv_0 = f_3\circ f_2\circ f_1$.
\end{lemma}

\begin{proposition}
    Given a countably infinite set $X$, any permutation $\sigma\in S_{fin}(X)$ and its inverse can be written as a composition of forgetful and retentive swaps each with unique domain.
\end{proposition}

\begin{theorem}
    Let $X$ be a countably infinite set, $\sigma \in S_{fin}(X)$. Then there exist $f_1$ and $f_2$, which are forgetful and retentive respectively, such that $dom(f_1)\neq dom(f_2)$ and $\siginv = f_2 \circ f_1$.
\end{theorem}


\subsection{Finite Disjoint Uses} \label{finite disjoint uses}
In this case, we will suppose a countable set of people enumerated by $\NN$ have used the machine. The machine does the swap $f: \NN \to \NN$ as $f(n) = n+1$. As in the finite case, we will introduce an outsider who has not used the machine. Let $z$ denote this outsider. We can undo $f$ in the following steps. 
\begin{multicols}{3}
\begin{itemize}
    \item[Step 1:] 
        \begin{align*}
            2 &\to 1 \\
            z &\to 2 \\
            3 &\to z \\
            4 &\to 3 \\
            5 &\to 4 \\
            \vdots 
        \end{align*}
        
    \item [Step 2:]
        \begin{align*}
            z &\to 2 \\
            2 &\to 3 \\
            3 &\to 4 \\
            4 &\to 5 \\
            5 &\to 6 \\
            \vdots  
        \end{align*}

    \item [Step 3:]
        \begin{align*}
            3 &\to z \\
            4 &\to 3 \\
            5 &\to 4 \\
            6 &\to 5 \\
            7 &\to 6 \\
            \vdots  
        \end{align*}
    
\end{itemize}
\end{multicols}


The original swap and map in step 1 can be defined as
    
\begin{align}f(x) = x + 1 \qquad
f_1(x) = \begin{cases}
    2 & \text{ if $x = z$}\\
    z & \text{ if $x = 3$}\\
    x-1 & \text{ else}.
\end{cases}
\end{align}
We can analogously define step 2 and step 3. The piecewise functions can be written more concisely using an extension of cycle notation. For example, 

\begin{equation}
    f = (1\; 2\; 3\; 4 \cdots) \qquad f_1 = (\cdots4\; 3\; x\; 2\; 1).
\end{equation}
In the cycle notation, each person has their mind sent to the person to their right. Notice that forgetful swaps trail off on the right while retentive ones trail off on the left. This notation makes the proof of the following lemma much easier to write:

\begin{lemma}
    Let $f_0: \NN \to \NN$ defined as $f_0(n) = n+1$. Then there exist $f_1$, $f_2$, and $f_3$ where $f_1$, $f_3$ are retentive and $f_2$ is forgetful, such that $dom(f_i)\neq dom(f_j)$ for all $i\neq j$, and $\finv_0 = f_3\circ f_2\circ f_1$.
\end{lemma}

\begin{proof}
    Define 
    \begin{align*}
        f_1 &= (\cdots 4\; 3\; x\; 2\; 1)\\
        f_2 &= (x\; 2\; 3\; 4\cdots)\\
        f_3 &= (\cdots5\; 4\; 3\; x).
    \end{align*}

    Then $dom(f_0) = \NN$, $dom(f_1) = \NN \cup \{z\}$, $dom(f_2) = dom(f_1)\setminus\{1\}$, and $dom(f_3) = dom(f_1)\setminus\{1,2\}$. By tracing elements, we see that $\finv_0 = f_3\circ f_2\circ f_1$.
\end{proof}

We can generalize the swap as follows. Let $X$ denote a countably infinite set of people, defined as $X =\{a_1, a_2, \cdots \}$, that uses the machine. The machine does the swap defined by the function $f: X\to X$, $f(a_n) = a_{n+1}$. Define 

    $$\fhat = (\cdots a_5\; a_4\; a_3\; x)(x\; a_2\; a_3\; a_4\cdots)(\cdots a_4\; a_3\; x\; a_2\; a_1)$$
 and from the lemma, it is immediate to show $f^{-1} = \fhat$.\\

\begin{corollary}
    Let $X_1,\cdots,X_n$ be countably infinite, pairwise disjoint subsets of $X$ where $X_i = \{x_{i,j}\}_{j \in \NN}$.  Define $f: X \to X$ by $f(x) = g_1 \circ \cdots \circ g_n$ where 
    $$g_i(x) = (x_{i,1}\; x_{i,2}\; x_{i,3}\cdots) = 
    \begin{cases}
        x & \text{ if $x \notin X_i$}\\
        x_{i,j+1} & \text{ if $x\in X_i$}
    \end{cases}$$ 
    Then $f^{-1} = \ghat_n \circ \cdots \circ \ghat_1$.
\end{corollary}

\subsection{Finitary Permutations}\label{finitary permutations}
Let $X$ be a countably infinite set, and $S(X)$ the permutation group on $X$. Let $S_{fin}(X)$ be the subgroup of $S(X)$ consisting of permutations of $X$ which fix all but finitely many elements of $X$. We call $\sigma\in S(X)$ a \textit{finitary} permutation if $\sigma \in S_{fin}(X)$. As it turns out, we can obtain any finitary permutation from swaps using the machine under our restrictions. We start with cycles.
\begin{lemma}
    Let $X$ be a countably infinite set, $a_1,\cdots,a_n\in X$. Then there exists $f_1$ and $f_2$, forgetful and retentive respectfully, such that $dom(f_1)\neq dom(f_2)$ and $(a_1\; \cdots\; a_n) = f_2\circ f_1$.
\end{lemma}

\begin{proof}
    Let $f = (a_1\; a_2 \cdots a_n)$, then 
    $$f = (\cdots a_{n+2}\; a_{n+1}\; a_n)(a_1\; a_2\cdots a_n\; a_{n+1}\cdots)$$ 
\end{proof}

\begin{example}
    
Let $f= (a_1\; a_2)$ and $z \notin X$. Then the transposition is solved by: 

\begin{multicols}{2}
\begin{itemize}
    \item[Step 1:] 
        \begin{align*}
            a_1 &\to a_2 \\
            a_2 &\to z \\
            z &\to a_3 \\
            a_3 &\to a_4 \\
            a_4 &\to a_5 \\
            \vdots 
        \end{align*}
        
    \item [Step 2:]
        \begin{align*}
            z &\to a_1 \\
            a_3 &\to z \\
            a_4 &\to a_3 \\
            a_5 &\to a_4 \\
            a_6 &\to a_5 \\
            \vdots  
        \end{align*}
\end{itemize}
\end{multicols}

In cycle notation,

$$(a_1\; a_2)^{-1} = (a_1\; a_2\; z\; a_3\; a_4\cdots)(\cdots a_5\; a_4\; a_3\ z\ a_1)$$
\end{example}

\begin{lemma}
    Let $X$ be a countably infinite set, $a_1,\cdots,a_n\in X$, and $z\notin X$. Then there exists $f_1$ and $f_2$, forgetful and retentive respectfully, such that $z\in dom(f_1), dom(f_2)$, $dom(f_1)\neq dom(f_2)$, and $(a_1\; \cdots\; a_n)^{-1} = f_2\circ f_1$. 
\end{lemma}

\begin{proof}
    $(a_1\; a_2\cdots a_k)^{-1} = (\cdots a_{k+2}\ a_{k+1}\ z\ a_1)(a_1\; a_k\ a_{k-1}\cdots a_3\; a_2\; z\; a_{k+1}\ a_{k+2}\cdots)$
\end{proof}

Let $\sigma \in S_{fin}(X)$ be the cycle $\sigma = (a_1\; a_2\cdots a_k)$ and define
$$\sigstar = (\cdots a_{k+2}\ a_{k+1}\ z\ a_1)(a_1\; a_k\; a_{k-1}\cdots a_3\; a_2\; z\; a_{k+1}\; a_{k+2}\cdots)$$\\
Then from the above lemma, $\sigstar = \siginv$.

\begin{proposition}
    Given a countably infinite set $X$, any permutation $\sigma\in S_{fin}(X)$ and its inverse can be written as a composition of forgetful and retentive swaps each with unique domain.
\end{proposition}

\begin{proof}
    Since $\sigma\in S_{fin}(X)$, there exist disjoint cycles $\tau_1,\cdots\tau_m \in S_{fin}(X)$, such that $\sigma = \tau_1\cdots\tau_m$. Then $\siginv = \taustar_m\cdots\taustar_1$. Then we are done since the fact that each has a unique domain follows from the cycles being disjoint.
\end{proof}

Since each $\taustar_i$ uses two swaps on the machine, the total number of swaps needed to write $\siginv$ is $2m$. We can do better. Consider $\sigma = (a_1\; a_2)(a_3\; a_4\; a_5)$. Applying our solution from before we know $(a_1\; a_2)^{-1} = (a_1\; a_2\; z\; a_3\; a_4\cdots)(\cdots a_5\; a_4\; a_3\; z\; a_1)$. Examine the left permutation: $(a_1\; a_2\; z\; a_3\; a_4\; a_5 \cdots).$ For fixing $(a_1\; a_2)$, the only part of this permutation that is important is $a_1\to a_2\to z$. So we can change it to $(a_1\; a_2\; z\; a_5\; a_4\; a_3\cdots)$. Notice this fixes $a_4$ and $a_3$ along with our intention to fix $a_2$. We can generalize this idea to obtain a much more efficient solution.

\begin{example}
$\sigma = (a_1\; a_2)(a_3\; a_4\; a_5)$

Let $z\notin X$. Then the problem is solved by:

\begin{multicols}{2}
\begin{itemize}

    \item[Step 1:] 
        \begin{align*}
            a_1 &\to a_2 \\
            a_2 &\to z \\
            z &\to a_5 \\
            a_5 &\to a_4 \\
            a_4 &\to a_3 \\
            a_3 &\to a_6 \\
            a_6 &\to a_7 \\
            \vdots
        \end{align*}
        
    \item [Step 2:]
        \begin{align*}
            z &\to a_1 \\
            a_5 &\to z \\
            a_6 &\to a_5 \\
            a_7 &\to a_6 \\
            \vdots  
        \end{align*}
    
\end{itemize}
\end{multicols}
In cycle notation:
$((a_1\; a_2)(a_3\; a_4\; a_5))^{-1} = (\cdots a_7\; a_6\; a_5\; z\; a_1)(a_1\; a_2\; z\; a_5\; a_4\; a_3\; a_6\; a_7\cdots)$.\\
\end{example}

Generalizing this solution allows us to fix any finitary permutation in just 2 steps.
We create some notation for an arbitrary cycle: 

Let $\cyclist{a}{k} := a_1\ \cdots a_k$ and $\cyclinv{a}{k} := a_k\ \cdots a_1$ we write 

\begin{align*}
    &(a_1\cdots a_k) = (\cyclist{a}{k})
    &(a_k \cdots a_1) = (\cyclinv{a}{k})
\end{align*}
Note that $(\cyclist{a}{k})^{-1} = (\cyclinv{a}{k})$.

\begin{theorem}
    Let $X$ be a countable infinite set, $\sigma \in S_{fin}(X)$. Then there exist $f_1$ and $f_2$, which are forgetful and retentive respectively, such that $dom(f_1)\neq dom(f_2)$ and $\siginv = f_2 \circ f_1$.
\end{theorem}

\begin{proof}
    Let $\sigma = \tau_1\cdots\tau_m$ be the cycle decomposition. Let $\tau_i = (\cyclist{a_i}{k_i})$. 
    Then $$\sigma = (\cyclist{a_1}{k_1})(\cyclist{a_2}{k_2})\cdots(\cyclist{a_m}{k_m})$$. And thus by tracing elements one can see:
    $$\siginv = (b_3\ b_2\ b_1\ a_{m,k_m}\cdots a_{2,k_2}\ x\ a_{1,k_1})(\cyclinv{a_1}{k_1}\ x\ \cyclinv{a_2}{k_2}\cyclinv{a_3}{k_3}\cdots\cyclinv{a_m}{k_m}\ b_1\ b_2\ b_3\cdots)$$
where the $b_j \in X$ not permuted by any $\tau_i$ for any $i$.
\end{proof}

\section{Future Directions}
In our finite case, we only proved optimally for when $n =3$ so a natural extension is to prove optimally for any $n \in \NN$. A noteworthy observation in the $3$ cycle case, each product in the solution is a distinct $3$ cycle. As an example, consider the solution for when the machine permutes $3$ people. 
\begin{equation}
(1\;2\;3) = (3\;1\;x_1)(2\;3\;x_1)(1\;2\;x_1)
\end{equation}
Since each $3$ cycle is distinct, then they correspond to a unique triangle on the graph with vertices $\{1, 2, 3, x_1\}$. If we count all possible distinct triangles that have at least $x_1$ in the graph, this places an lower bound for the the total number of moves required. 
It would be interesting to see if these ideas carry over to more general cases. \\
In the infinite case we lost a group structure when we considered a machine that does $n+1$ and $n-1$ on $\NN$. It would be interesting to explore the scenario with a group structure by considering a machine that does $n+1$ and $n-1$ on $\ZZ$. One could also look at a machine that preforms a different permutation on a countable set. For a countably infinite set $X$, one could consider a finite machine that only cycles $m$-people.

\newpage

\nocite{*}
\bibliographystyle{plain} 
\bibliography{keeler} 

\newpage 


\newpage

\end{document}